\documentclass[a4paper, reqno, 12pt]{amsart}

\usepackage[utf8]{inputenc}
\usepackage{amsthm,amsfonts,amssymb,amsmath,amsxtra,amsrefs}
\usepackage{mathtools}
 \usepackage{bbold}
\usepackage{latexsym}
 \usepackage{stmaryrd}
\usepackage{graphicx}
\usepackage{tikz}
\usetikzlibrary{positioning}
\usepackage{tikz-cd}
\usepackage{todonotes,cancel}
\usepackage[margin=1in]{geometry}
\usepackage{mathrsfs}
\usepackage{bm}

\usepackage[all]{xy}
\SelectTips{cm}{}
\usepackage{xr-hyper}
\usepackage[colorlinks=false,
   citecolor=Black,
   linkcolor=Red,
   urlcolor=Blue]{hyperref}
\usepackage{verbatim}
\RequirePackage{xspace}
\RequirePackage{etoolbox}
\RequirePackage{varwidth}
\RequirePackage{enumitem}
\RequirePackage{tensor}
\RequirePackage{mathtools}
\RequirePackage{longtable}
\RequirePackage{multirow}

\usepackage{rotating}

\tikzset{
    labl/.style={anchor=south, rotate=90, inner sep=.5mm}
}

\newtheorem{thm}{Theorem}

\newtheorem{thmintro}{Theorem}

\newtheorem{prop}[thm]{Proposition}

\newtheorem{lem}[thm]{Lemma}

\newtheorem{cor}[thm]{Corollary}

\theoremstyle{definition}

\newtheorem{defi}[thm]{Definition}

\newtheorem{rem}[thm]{Remark}
\newtheorem{remark}[thm]{Remark}

\numberwithin{equation}{section}
\numberwithin{thm}{section}

\def\ge{\geqslant}
\def\le{\leqslant}

\def\i{^{-1}}

\def\<{\langle}
\def\>{\rangle}

\newcommand{{\BG}}{\ensuremath{\mathbb {G}}\xspace}

\newcommand{{\BK}}{\ensuremath{\mathbb {K}}\xspace}

\newcommand{\BR}{\ensuremath{\mathbb {R}}\xspace}

\newcommand{\CB}{\ensuremath{\mathcal {B}}\xspace}

\newcommand{\CP}{\ensuremath{\mathcal {P}}\xspace}

\begin{document}

\title[]{Acyclic matchings on Bruhat intervals and applications to total positivity}

\author[Huanchen Bao]{Huanchen Bao}
\address{Department of Mathematics, National University of Singapore, Singapore.}
\email{huanchen@nus.edu.sg}

\author[Xuhua He]{Xuhua He}
\address{Department of Mathematics and New Cornerstone Science Laboratory, The University of Hong Kong, Pokfulam, Hong Kong, Hong Kong SAR, China}
\email{xuhuahe@hku.hk}


\begin{abstract}
The existence of acyclic complete matchings on the face poset of a regular CW complex implies that the underlying topological space of the CW complex is contractible by discrete Morse theory. 

In this paper, we construct explicitly acyclic complete matchings on any non-trivial Bruhat interval $[v,w]$ based on any reflection order on the Coxeter group $W$. We then apply this combinatorial result to regular CW complexes arising from the theory of total positivity. 

As an application, we show that the totally nonnegative Springer fibers are contractible. This verifies a conjecture of Lusztig in \cite{Lu94}. As another application, we show that the totally nonnegative fibers of the natural projection from full flag varieties to partial flag varieties are contractible. This leads to a much simplified proof of the regularity property on totally nonnegative partial flag varieties compared to the proofs in \cite{GKL} and \cite{BH22}. 
\end{abstract}

	\maketitle
	
	\tableofcontents
\section{Introduction}

    \subsection{Acyclic complete matchings}Discrete Morse theory, developed by Forman in \cite{Forman}, is an efficient tool for determining the homotopy type of a regular CW complex. The theory was reformulated by Chari \cite{Chari} in purely combinatorial terms of acyclic matchings on the face poset $P$ of the regular CW complex, and was later applied by Rietsch and Williams \cite{RW} to establish the regularity of totally nonnegative flag varieties up to homotopy equivalence. Chari's reformulation applies to arbitrary posets; see \S\ref{sec:discrete} for the definition.

    The poset we are mainly interested in arises from the Coxeter groups. We state the main combinatorial result of this paper. 
    
   \begin{thmintro}[Theorem~\ref{thm:M}]
    A reflection order on a Coxeter group $W$ induces an acyclic complete matching on any non-trivial Bruhat interval $[v, w]$ for $v<w$ in $W$. 
   \end{thmintro}

    It is known \cite{Bj1} that $[v,w]$ is the augmented face poset of a regular CW complex homeomorphic to a closed ball. So discrete Morse theory is not very useful on the poset $[v,w]$ itself. However, there are many regular CW complexes that are not homeomorphic to closed balls and are not even  equidimensional. The explicit matching from Theorem ~\ref{thm:M} will be a useful tool to handle some of these complicated situations.

    Let us mention some previous works in this direction. Chari \cite{Chari} showed that a shelling of a poset implies the existence of a desired matching. Dyer \cite{Dyer1} showed that a reflection order on $W$ implies the existence a shelling on Bruhat intervals. Combining the two results, Rietsch and Williams \cite{RW}*{Corollary~8.2} obtained that the nontrivial Bruhat interval $[v,w]$ admits an acyclic complete matching induced from any reflection order. It is important to point out that such matching were constructed in a recursive way. It was subsequently studied by Jones \cite{Jones}. 

    Our construction, on the other hand, is explicit and non-recursive. Such explicit construction has the advantage to handle some complicated posets, for example, certain subsets of the Bruhat intervals, which serve as the face posets of some interesting geometric spaces we will consider in this paper. We also show that the matching from our construction coincides with the one in \cite{RW}.

\subsection{Totally positive spaces} The regular CW complexes we consider in this paper arise from the theory of total positivity. 

The theory of total positivity on the reductive groups and their flag varieties was introduced by Lusztig in the seminal work \cite{Lu94}, and later generalized to arbitrary Kac-Moody types in \cite{Lu19} and \cite{BH21}. The totally nonnegative flag variety $\CP_{K, \ge 0}$ is a ``remarkable polyhedral subspace'' (cf. \cite{Lu94}). It admits a cellular decomposition and is a regular CW complex homeomorphic to a closed ball (see \cite{Rie06}, \cite{GKL}, \cite{BH22}). The totally nonnegative flag varieties have been used in many other areas, such as cluster algebras \cite{FZ}, Grassmann polytopes \cite{Lam15}, and the physics of scattering amplitudes \cite{AHBC16}. 


The regular CW complex structure for the totally nonnegative full flag variety is given by $\CB_{\ge 0}=\sqcup_{v \le w} \CB_{v, w; >0}$, where $\CB_{v, w; >0}$ is the totally nonnegative Richardson variety of $\CB$. In this paper, we will establish the contractibility of certain subcomplexes inside the totally nonnegative flag variety $\CB_{\ge 0}$. This will be established by constructing explicit acyclic matchings on certain subposets of $\{(v, w) \in W^2 \vert v \le w\}$. These matchings will be obtained via carefully chosen reflection orders on $W$ via Theorem ~\ref{thm:M}.

The first subcomplex of $\CB_{\ge 0}$ we study is the totally nonnnegative Springer fiber for a reductive group. By definition, a totally nonnnegative Springer fiber is $\CB^u_{\ge 0}=\{B \in \CB_{\ge 0} \vert u \cdot B=B\}$ for some unipotent $u \in G_{\ge 0}$.  Lusztig \cite{Lu94}*{\S8.16} conjectured that totally nonnegative Springer fibers are contractible. In \cite{LuSp}, Lusztig showed that $\CB^u_{\ge 0}$ admits a natural cellular decomposition, inherited from the cellular decomposition of $\CB_{\ge 0}$ and gave an explicit description of the cells inside $\CB^u_{\ge 0}$. As a subcomplex of $\CB_{\ge 0}$, one deduces that $\CB^u_{\ge 0}$ is a regular CW complex by the regularity theorem of Galashin-Karp-Lam \cite{GKL}. However, the subcomplex $\CB^u_{\ge 0}$ is not equidimensional in general.

\begin{thmintro}[Theorem~\ref{thm:Sp}]
    The totally nonnegative Springer fiber $\CB^u_{\ge 0}$ is contractible.
\end{thmintro}

The second subcomplex of $\CB_{\ge 0}$ we study arise from the fiber of the projection map from the totally nonnegative full flag variety to the totally nonnegative partial flag variety for any Kac-Moody group. 
    
The totally nonnegative partial flag variety $\CP_{K, \ge 0}$ admits the natural cellular decomposition into totally nonnegative projected Richardson varieties $\CP_{K, v,w, \ge 0}$ \cites{Rie99, Rie06}, inherited from the cellular the cellular decomposition of $\CB_{\ge 0}$. By definition, the closed stratum $\CP_{K, v,w, \ge 0}$ is the image of the totally nonnegative Richardson variety $\CB_{v,w, \ge0}$ in the full flag variety. The natural projection map $\pi_{v,w}: \CB_{v,w, \ge0} \rightarrow \CP_{K, v,w, \ge 0} $ is almost never a homeomorphism. We show that

\begin{thmintro}
    The fibers of the map $\pi_{v,w}: \CB_{v,w, \ge0} \rightarrow \CP_{K, v,w, \ge 0} $  are contractible regular CW complexes.
\end{thmintro}
    
We conclude that the map $\pi_{v,w}$ is cell-like in the sense of \cite{Dyd02} or \cite{DHM}*{Definition~2.34}. As the name suggests, one should think of cell-like maps as well-behaved maps between CW complexes. Indeed, Siebenmann's cell-like approximation theorem implies that the two spaces $\CB_{v,w, \ge0}$ and $\CP_{K, v,w, \ge 0}$ are abstractly homeomorphic; cf. \cite{DHM}*{Corollary~2.33}. Therefore, the regularity of $\CB_{v,w, \ge0}$ would readily imply the regularity of $\CP_{K, v,w, \ge 0}$. 
    
This leads to a new proof of the regularity property for the totally nonnegative partial flag varieties, assuming the regularity theorem on the full flag variety. This approach significantly simplifies the proof of the regularity theorems for totally nonnegative partial flag varieties by Galashin-Karp-Lam \cite{GKL} for reductive groups and by us \cite{BH22} for arbitrary Kac-Moody groups, which involves the compatibility of certain total positivity with the atlas map (see, e.g., \cite{BH22}*{\S6-\S8}).

    \vspace{.2cm}
\noindent {\bf Acknowledgement: } 
XH is partially supported by the New Cornerstone Science Foundation through the New Cornerstone Investigator Program and the Xplorer Prize, and by Hong Kong RGC grant 14300023. HB is supported by MOE grants A-0004586-00-00 and A-0004586-01-00. HB thanks Lauren Williams for introducing him to discrete Morse theory. 
     
\section{Discrete Morse theory}\label{sec:combinatorics}
In this section, we recall the discrete Morse theory for regular CW complexes by Forman \cite{Forman}. We focus on the combinatorial reformulation by Chari in \cite{Chari} in terms of acyclic matching on the face poset. 
    
\subsection{Some properties on  posets}\label{subsec:posets}
 Let $(P, \le)$ be a poset. For any $x, y \in P$ with $x \le y$, the interval from $x$ to $y$ is defined to be $[x, y]=\{z \in P \vert x \le z \le y\}$. 
The covering relations are indicated by $\gtrdot$ and $\lessdot$. For any $x, y \in P$ with $x \le y$, a maximal chain from $x$ to $y$ is a finite sequence of elements $y=w_0 \gtrdot w_1 \gtrdot \cdots \gtrdot w_n=x$. In general, the maximal chain may not exist. 

The poset $P$ is called {\it pure} if for any $x, y \in P$ with $x \le y$, the maximal chains from $x$ to $y$ always exist and have the same length. A pure poset $P$ is called {\it thin} if every interval of length $2$ has exactly $4$ elements, i.e. has exactly two elements between $x$ and $y$. A pure poset is called {\it graded} if it has a unique minimal and maximal element. We often denote the unique minimal element (resp. maximal element) of a graded poset by $\hat{0}$ (resp. $\hat{1}$). The coatoms (resp. atoms) of a graded poset are the elements covered by $\hat{1}$ (resp. $\hat{0}$).

Assume $P$ is a CW poset, that is, $P$ is the face poset of a regular CW complex $\Delta(P)$. The topology of $\Delta(P)$ is determined by the poset $P$. Hence we say that $P$ is the face poset of a ball (resp., sphere) if the CW complex $\Delta(P)$ is homeomorphic to a ball (resp., a sphere.)

Now we recall the notion of EL-shellability introduced by Bjorner in \cite{Bj1}. Suppose that the poset $P$ is pure. An edge labeling of $P$ is a map $\lambda$ from the set of all the covering relations in $P$ to a poset $\Lambda$. Labeling $\lambda$ sends any maximal chain of an interval of $P$ to a tuple of $\Lambda$. A maximal chain is called {\it increasing} if the associated tuple of $\Lambda$ is increasing. 
An edge labeling of $P$ is called {\it EL-labeling} if for  every interval, there exists a unique increasing maximal chain, and all the other maximal chains of this interval are less than this maximal chain with respect to the lexicographical order.

\subsection{Discrete Morse theory}\label{sec:discrete}
    Let $G(P)$ be the Hasse diagram of $P$ with the edges being the covering relations. A {\it matching} of the Hasse diagram of $G(P)$ (or simply a matching of $P$) is a subset of edges of $G(P)$ such that any vertex $x \in P$ is incident to at most one edge in $M$. We say that a matching $M$ is {\em complete} if there is no unmatch element. We often equivalently consider $M$ as an involution of the set  $P$ of vertices, such that $M(x) = y$ if the edge $\{x,y\} \in M$ and $M(x) =x$ if $x$ is unmatched.
    
     We can regard the Hasse diagram $G(P)$ of $P$ as a directed graph with edges directed from large elements to small elements with respect to poset ordering. By reversing the directed edges in $M$, we obtain a new directed graph $G_M(P)$. We say that $M$ is an {\it acyclic matching (or discrete Morse matching)} if $G_M(P)$ has no cycles. 
    
    \begin{defi}
    Let $M$ be a matching of $P$, considered as an involution of $P$. We say that a subset $Q \subset P$ is an {\em $M$-subset}, if $M$ restricts to an involution of $Q$. In other words, there is no edge in $M$ connecting any element in $Q$ with any element in $P-Q$. 

    It is clear that the collection of $M$-subsets is preserved under unions, intersections, and subtractions. It is also clear that if $M$ is acyclic, then any $M$ subsets are acyclic.
    \end{defi}
    
    The following reformulation of Forman's discrete Morse theory on regular CW complexes by Chari \cite{Chari} is crucial for us. Our formulation follows \cite{RW}*{Theorem~7.6}.

\begin{thm}\label{thm:Chari} 
    Let $P$ be the face poset of a regular CW complex $\Delta(P)$ with an acyclic matching $M$. Let $m_p$ be the number of unmatched cells of dimension $p$. Then $\Delta(P)$ is homotopic to a (not necessarily regular) CW complex of exactly $m_p$ cells in dimension $p$. 

    In particular, if $m_0 =1$ and $m_p = 0$ for all $p >0$, we can conclude that $\Delta(P)$ is contractible.
\end{thm}

If $P$ is the face poset of a regular CW complex $\Delta(P)$, we denote by $\hat{P} = P \sqcup \{\hat{0}\}$ the augmented face poset with an additional minimal element $\hat{0}$. Assuming that $P$ has an acyclic matching, we define an extension of the acyclic matching by matching $\hat{0}$ with any unmatched (if exists) minimal element of $P$. The resultant matching is clearly acyclic. 

\begin{rem}
Our definition of face poset $P$ of a regular CW complex agrees with \cite{Forman} and differs from \cites{Bj1, Bj2}. In particular, this is no added minimal element $\{\hat{0}\}$ corresponding to the empty set. One should think of the difference as ordinary homology versus reduced homology.  
\end{rem}

 \section{Acyclic matchings on Bruhat intervals}\label{subsec:M}
 
Let $(W,I)$ be a Coxeter group with simple reflections $s_i \in  I$. We denote the Bruhat order on $W$ by $v \le w $ and the cover relation by $v \lessdot w$, respectively. We denote the length function by $\ell(\cdot)$. Let $T$ be the set of reflections of $W$ and denote an (arbitrary fixed) reflection order by $\preceq$; cf. \cite{Dyer1}*{\S2}. We shall consider various different reflection orders, which will be specified in the context. 

 \subsection{EL-labelings} \label{subsec:ELorder}
 For any $v\le w$, let $[v,w] = \{z \in W\vert v \le z \le w\}$ be the Bruhat interval. By \cite{BB}*{Lemma~2.7.3}, the poset $[v,w]$ is thin.  For any $w_1 \lessdot w_2$, we label the edge by 
\begin{equation}\label{eq:EL}
\lambda(w_1 \lessdot w_2) = w_1 w_2^{-1} \in T.
\end{equation} 
It follows from \cite{Dyer1} that this defines an EL-labeling of $[v,w]$ with respect to the (any fixed) reflection order $\preceq$ of $T$.
Recall that the opposite order $\preceq^{op}$ on $T$ is again a reflection order. By \cite{Dyer1}, the edge labeling \eqref{eq:EL} is a dual EL-labeling of $[v,w]$ (that is, an EL-labeling for the dual poset), with respect to the opposite reflection order $\preceq^{op}$ of $T$. 

Let $v \lessdot  v_1 \lessdot  \cdots \lessdot  w_1 \lessdot  w$ be the unique increasing maximal chain in $[v,w]$ with respect to the reflection order $\preceq$ on $T$.  We then obtain the following conclusions combining both the EL-labeling and the dual El-labeling of $[v,w]$. 

(a) {\it  The maximal chain $w \gtrdot w_1 \gtrdot \cdots \gtrdot v_1 \gtrdot v$ is the unique decreasing maximal chain, which is lexicographically maximal.}

(b) {\it Let $\{v_1, \dots, v_n\}$ be the atoms of $v$. Then we have $ v v_1^{-1}\preceq  v v_i ^{-1}$, for any $1 \le i \le n$.} 

(c) {\it Let $\{w_1, \dots, w_m\}$ be the coatoms of $w$. Then we have $  w_j w^{-1} \preceq  w_1 w^{-1}$, for any $1 \le j \le m$.}

  \subsection{Acyclic matchings} \label{subsec:defM}

Let $G([v,w])$ be the Hasse diagram of $[v,w]$. We define a subset $M_{\preceq}$ of the edges of $G([v,w])$ with respect to the reflection order $\preceq$ on $T$ as follows. For any $x \in [v,w]$, let $E(x)$ be the set of edges in $G([v,w])$ incident to $x$. Recall that the edges in $E(x)$ are labeled by reflections in $T$ via the labeling function $\lambda: E(x) \rightarrow T$ in \eqref{eq:EL}. Labelings on $E(x)$ are distinct by definition.  Let $E(x,max)$ be the edge such that the labeling is maximal with respect to the order $\preceq$ in $T$. Let $M_{\preceq}([v,w]) = \cup_{x \in [v,w]}E(x,max)$ as a subset of edges of $G([v,w])$.

\begin{thm}\label{thm:M}
Let $v < w$. Then $M_\preceq([v,w])$ is an acyclic complete matching of $[v,w]$.
\end{thm}

The rest of this section is devoted to the proof of Theorem~\ref{thm:M}. We first show that $M_\preceq([v,w])$ is indeed a complete matching of $[v,w]$ in \S\ref{sec:matching}. We then discuss the relation between the matching $M_\preceq([v,w])$ and the shelling (induced by the EL-labeling) of $[v,w]$ in \S\ref{sec:shelling}. Finally, we prove that the matching $M_\preceq([v,w])$ is acyclic in \S\ref{sec:acyclic}. We write $M = M_\preceq([v,w])$.

\begin{rem}
Our matching in Theorem~\ref{thm:M} coincides with the matching constructed in \cite{RW}*{Corollary~8.2} thanks to Corollary~\ref{cor:shelling}. 
The explicit (and non-recursive) construction of such matchings allows us to study matchings on more complicated posets in \S\ref{sec:Sp} and \S\ref{sec:fibers}. 
\end{rem}

\subsubsection{}\label{sec:matching} We first show that $M$ is indeed a complete matching of $[v,w]$.
 
  We can clearly consider $M$ as a map from $[v,w]$ to $[v,w]$, such that \[
\{x, M(x)\} =E(x, max),\quad  \text{ for any } x \in [v,w].
\]
There is no unmatched element by definition. It remains to show that it is an involution.  

Let $x \in [v,w]$. Without loss of generality, we assume $x \lessdot M(x)$. We assume by contradiction that $M^2(x)  \neq x$. We divide into two cases.

We first consider the case $M(x) \lessdot M^2(x)$. By definition, we have $\lambda(x \lessdot M(x)) \preceq \lambda (M(x) \lessdot M^2(x))$. Consider the length-two interval $[x, M^2(x)]$, which consists precisely of $4$ elements, $x$, $M(x)$, $M^2(x)$, $y$. By assumption, $x \lessdot M(x) \lessdot M^2(x)$ is the unique increasing maximal chain in $[x, M(x)]$ in the EL-labeling with respect to the reflection $\preceq$ on $T$. So we must have $\lambda (x \lessdot M(x)) \preceq \lambda (x \lessdot y) $ by \S\ref{subsec:ELorder} (b). This contradicts the fact that $\lambda (x \lessdot M(x))$ is maximal. 

We then consider the case $M^2(x) \lessdot M(x)$. By definition, we have $\lambda(x \lessdot M(x))  \preceq \lambda (M^2(x) \lessdot M(x))$. We consider the interval $[v,x]$ and let $v \le v_1 \le \cdots \le x_1 \le x$ be the unique increasing maximal chain in the EL-labeling. Since $ \lambda (x_1 \preceq x) \preceq \lambda(x \lessdot M(x))$ by definition, the extended chain $v \le v_1 \le \cdots \le x_1 \le x \le M(x)$ is increasing. Hence it is the unique increasing chain in the interval $[v,w]$. Then by \S\ref{subsec:ELorder} (c), we must have $\lambda (M^2(x) \lessdot M(x)) \preceq \lambda(x \lessdot M(x))$. This is a contradiction.

\subsubsection{}\label{sec:shelling}
Let $[v',w'] \subset [v,w]$ be a sub-interval. If $[v',w']$ is an $M_{\preceq}([v,w])$-subset of $[v,w]$, then the restriction of $M_{\preceq}([v,w])$ is the same as $M_{\preceq}([v',w'])$, as involutions/matchings on $[v',w']$.

The next corollary discusses the relation between the matching $M_{\preceq}([v,w])$ of $[v,w]$ and the shelling of $[v,w]$.

\begin{cor}\label{cor:shelling}
Let $v < w$. 
\begin{enumerate} 
\item Let $w_1, \dots, w_n$ be the coatoms of $[v,w]$, ordered such that $ w w_1^{-1} \preceq \cdots \preceq  w w_n^{-1}$ in $T$. Then $\cup_{1 \le i < k} [v, w_i]$ are $M_{\preceq}([v,w]) $-subsets for any $k \le n$. 
\item Let $v_1, \dots, v_m$ be the atoms of $[v,w]$, ordered in such a way that $vv_1^{-1} \preceq \cdots \preceq  v v_n^{-1}$ in $T$. Then $\cup_{1 \le i < k} [v_i, w]$ are $M_{\preceq}([v,w]) $-subsets for any $k \le m$. 
\end{enumerate} 
\end{cor}

\begin{proof}We show part (1). Part (2) is similar and will be omitted.

We write $M = M_{\preceq}([v,w]) $ and consider $M$ as an involution on $[v,w]$. We need to show that for any $x \in \cup_{1 \le i < k} [v, w_i]$, we have $M(x) \in   \cup_{1 \le i < k} [v, w_i]$. If $M(x) \lessdot x$, this is trivial. 

Assume $x \lessdot M(x)$. By \cite{Dyer1}, the interval $[x,w]$ is EL-shellable with respect to the opposite reflection order $\preceq^{op}$ on $T$. Let $x \le x_1 \le \dots \le w' \le w$ be the unique increasing maximal chain with respect to the opposite reflection order $\preceq^{op}$ on $T$. Applying \S\ref{subsec:ELorder} (b) with respect to the order $\preceq^{op}$, we obtain that $\lambda(x \lessdot x_1)$ is minimal with respect to the order $\preceq^{op}$ among atoms of $x$, hence maximal with respect to the order $\preceq$. So $M(x) = x_1$. By \S\ref{subsec:ELorder} (c), we see that $w'$ must be $w_i$ for some $i <k$. This finishes the proof.
\end{proof}

\subsubsection{}\label{sec:acyclic}

We show that the matching $M = M_\preceq([v,w])$ is acyclic. 

We prove this by induction on $m = \ell(w) - \ell(v)$. When $m=1$, this is trivial. We assume $m\ge 2$. 

 Let $w_1, \dots, w_n$ be the coatoms of $[v,w]$, ordered such that $ ww_1^{-1} \preceq \cdots \preceq  w  w_n^{-1}$  in $T$. It follows by Corollary~\ref{cor:shelling} that $[v,w_1]$ is an $M$-subset of $[v,w]$. Therefore, the involution $M$ restricts to an involution of $[v,w_1]$. It is clear that this is the same as $M_{\preceq}([v,w_1])$. Hence $[v,w_1]$ is an acyclic $M$-subset by the induction hypothesis.

 (a) {\it  The subset $\cup_{1 \le i < n }[v,w_i]$ is an acyclic $M$-subset.}

Without loss of generality, we assume $n=3$ for notational simplicity. The general case follows by induction on $n$ via the same argument. 

 By Corollary~\ref{cor:shelling}, we know $[v,w_1] \cup [v,w_2]$ is an $M$-subset.  We consider the partition 
 \[
 [v,w_1] \cup [v,w_2] = [v,w_1] \sqcup P_2, \quad \text{where } P_2 = [v,w_2] - [v,w_1].
 \]
The subset $P_2 = [v,w_2] - [v,w_1]$ is also an $M$-subset. Then the restriction of $M$ on $P_2$ is the same as the restriction of $M_{\preceq}([v,w_2])$ by the construction in \S\ref{subsec:defM}. Since the matching $M_{\preceq}([v,w_2])$ of $[v,w_2]$ is acyclic by the induction hypothesis, the matching on the subset $P_2$ is also acyclic. 
Now we have shown that there are no cycles within either $[v,w_1]$ or $P_2$. On the other hand, any possible cover relation between $x \in P_2$ and $y \in [v, w_1]$ must be of the form $y \lessdot x$ by definition. So, there cannot be any cycle that traverses both $[v, w_1]$ and $P_2$. This proves the claim. 

(b) {\it We have $[v,w] - \cup_{1 \le i < n }[v,w_i] = [M(w), w]$.} 

Since the Bruhat interval $[v,w]$ is thin, any $y \le M(w)$ must be in $\cup_{1 \le i < n }[v,w_i]$. This proves (b).

Now we have a partition of $[v, w] = \Big(\cup_{1 \le i < n }[v,w_i] \Big) \sqcup [M(w), w]$. We can prove similarly to the proof of claim (a) that the matching $M$ on $[v, w]$ is acyclic.

We now finish the proof of Theorem~\ref{thm:M}.

\subsection{An example}\label{example:1}
We give an example of our matching. This example also shows that the matching is not always a special matching in the sense of \cite{BCM}. 

Let $S_4 = Perm\{1,2,3,4\}$ be the symmetric group, and let $s_1 = (12)$, $s_2 = (23)$, $s_3 = (34)$ be simple reflections. Fix a reflection order $\preceq$ on $T$ such that 
\[
s_1 \preceq s_1s_2s_1 \preceq s_1s_2s_3 s_2 s_1 \preceq s_2 \preceq s_2 s_3 s_2 \preceq s_3.
\]
Then the matching $M = M_{\preceq}([s_1, s_2s_3s_1s_2])$ on the interval $[s_1, s_2s_3s_1s_2]$ consists of the red edges. 
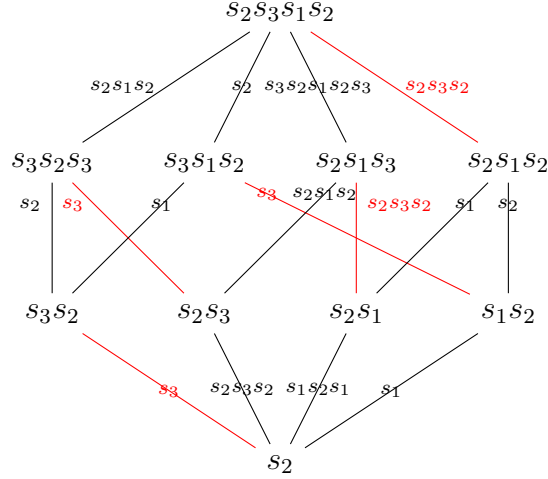
\[
\begin{tikzpicture}[node distance=1cm]

  \node (top) at (0,0) {$s_{2}s_3s_1s_2$};

  \node   (l1-1) at (-3,-2) {$s_3s_2s_3$};
  \node  (l1-2) at (-1,-2) {$s_3s_1s_2$};
    \node  (l1-3) at (1,-2) {$s_2s_1s_3$};
    \node  (l1-4) at (3,-2) {$s_2s_1s_2$};
  \node   (l2-1) at (-3,-4) {$s_3s_2$};
  \node  (l2-2) at (-1,-4) {$s_2s_3$};
    \node  (l2-3) at (1,-4) {$s_2s_1$};
    \node  (l2-4) at (3,-4) {$s_1s_2$};
  \node   (l3-1) at (-0,-6) {$s_2$};


  \draw (top) --  node[pos=0.5,left] {$\scriptstyle s_2s_1s_2$}  (l1-1);
  \draw (top) -- node[pos=0.5] {$\scriptstyle  s_2 $}  (l1-2);
\draw (top) --node[pos=0.5] {$\scriptstyle  s_3s_2 s_1s_2s_3$} (l1-3);
    \draw (top)[red] --node[pos=0.5, right] {$\scriptstyle  s_2 s_3s_2$} (l1-4);

  \draw (l1-1)  --node[pos=0.2,left] {$\scriptstyle  s_2$} (l2-1);
    \draw (l1-1)[red] --node[pos=0.2,left] {$\scriptstyle  s_3$} (l2-2);

  \draw (l1-2) -- node[pos=0.2] {$\scriptstyle  s_1$} (l2-1);
    \draw (l1-2)[red] --node[pos=0.1] {$\scriptstyle  s_3$} (l2-4);

  \draw (l1-3) --node[pos=0.1] {$\scriptstyle  s_2s_1s_2$} (l2-2);
    \draw (l1-3)[red] --node[pos=0.2, right] {$\scriptstyle  s_2s_3s_2$} (l2-3);

      \draw (l1-4) --node[pos=0.2] {$\scriptstyle  s_1$} (l2-3);
    \draw (l1-4) --node[pos=0.2] {$\scriptstyle  s_2$} (l2-4);

    \draw (l2-1)[red] -- node[pos=0.5] {$\scriptstyle  s_3$} (l3-1);
    \draw (l2-2) -- node[pos=0.5] {$\scriptstyle  s_2s_3s_2$} (l3-1);
    \draw (l2-3) --node[pos=0.5] {$\scriptstyle  s_1s_2s_1$} (l3-1);
    \draw (l2-4) --node[pos=0.5] {$\scriptstyle  s_1$} (l3-1);

\end{tikzpicture}
\]
Note that $s_3 s_2 s_3 \lessdot s_2 s_3 s_1 s_2$, but $M(s_3 s_2 s_3) \not \le M(s_2 s_3 s_1 s_2) $.



    \section{Totally nonnegative Springer fibers} \label{sec:Sp}
        
Our first application is for totally non-negative Springer fibers and the posets arising from them. 

\subsection{Totally nonnegative flag varieties}\label{sec:fl}
We first give a quick review of the totally nonnegative flag varieties. Let $G$ be a split Kac-Moody group over $\mathbb{R}$. Fix a pinning $(T, B^+, B^-, x_i, y_i; i \in I)$ of $G$. Let $(W,I)$ be the Weyl group of $G$. For any $K \subset I$, we denote by ${}^K W$ (resp. $W^K$) the set of left (resp. right) coset minimal-length representatives. Let $\CB=G/B^+$ be the (thin) full flag variety of $G$. It admits a stratification into the (open) Richardson varieties $$\CB=\sqcup_{v \le w} \CB_{v, w}, \text{ where } \CB_{v, w}=B^+ \dot w \cdot B^+ \cap B^- \dot v \cdot B^+.$$

Let $G_{\ge 0}$ be the totally nonnegative submonoid defined by Lusztig in \cites{Lu94, Lu19}.  The totally nonnegative full flag variety $\CB_{\ge 0}$ is defined to be the (Hausdorff) closure of $G_{\ge 0} \cdot B^+$ in $\CB$. The totally positive Richardson varieties are defined by $\CB_{v, w, >0}=\CB_{\ge 0} \cap \CB_{v, w}$. Let $\CB_{v, w, \ge 0}$ be the Hausdorff closure of $\CB_{v, w, >0}$. Then we have $\CB_{\ge 0} = \sqcup_{v\le w} \CB_{v,w,> 0}$. 

It is known that for any $v \le w$,   

\begin{itemize}
\item $\CB_{v, w, >0}$ is a semi-algebraic cell isomorphic to $\BR_{>0}^{\ell(w)-\ell(v)}$ (see \cite{Rie99}, \cite{BH21});

\item $\CB_{v, w, \ge 0}$ equals $\sqcup_{v \le v' \le w' \le w} \CB_{v', w', >0}$ (see \cite{Rie06}, \cite{BH22});

\item $\CB_{v, w, \ge 0}$ is a regular CW complex homeomorphic to a closed ball (see \cite{GKL}, \cite{BH22}). 
\end{itemize}

\subsection{Total positivity on Springer fibers}\label{sec:Springer}
In the rest of this section, we assume furthermore that $G$ is a reductive group, i.e. a Kac-Moody group of finite type. Then $W$ is a finite Coxeter group. We denote the longest element by $w_0$. For any $J \subset I$, we denote the longest element of the parabolic subgroup $W_J$ by $w_J$.

Let $u \in G_{\ge 0}$ be a unipotent element. By \cite{Lu94}, one may associate with $u$ two subsets $J$ and $J'$ of $I$ with $J \cap J'=\emptyset$. Let $\CB^u_{\ge 0}=\{B \in \CB_{\ge 0}\vert u \cdot B=B\}$ be the totally nonnegative Springer fiber of $u$. Lusztig proved in \cite{LuSp} that 
\[
     \CB^u_{\ge0}  = \sqcup_{z \in Z_{J,J'}} \CB_{v,w, >0}, 
\]
where 
    \begin{equation}\label{eq:ZJJ'}
    \begin{split}
     Z_{J, J'} = \{(v,w) \in W \times W \vert & v \le w; s_i w\le w, v \not \le s_iw,  \text{ for any } i \in J;\\
    & v \le s_jv, s_j v \not \le w, \text{ for any } j \in J'\}.
     \end{split}
    \end{equation}
Then $Z_{J, J'}$ is naturally a poset by with $(v, w) \le (v',w')$ if and only if $v' \le v '\le w' \le w$. It follows from the regularity of $\CB_{\ge 0}$ that $\CB^u_{\ge0}$ is also a regular CW complex. However, $\CB^u_{\ge 0}$ is not equidimensional, in general. 

Now we state the main result of this section. 
    
\begin{thm}\label{thm:ZJJ'}
Suppose that $W$ is a finite Coxeter group. Then the poset $Z_{J, J'}$ admits an acyclic matching with a unique unmatched element $(w_{ J'} w_0, w_{ J'} w_0)$.
\end{thm}

Combining Theorem~\ref{thm:ZJJ'} with Theorem~\ref{thm:Chari}, we establish the contractibility of the totally nonnegative Springer fiber, conjectured by Lusztig in \cite{Lu94}*{\S8.16}. 

\begin{thm}\label{thm:Sp}
The totally nonnegative Springer fiber 
$\CB^u_{\ge 0} $ is contractible. 
\end{thm}

\begin{remark}
Lusztig also conjectured that the totally nonnegative Grothendieck fiber $\CB^g_{\ge 0}$ for any $g \in G_{\ge 0}$ is contractible. This conjecture would follow from Theorem \ref{thm:Sp} together with the conjecture on the Jordan decomposition of $G_{\ge 0}$ \cite{HL1}*{Conjecture~3.1}. 
\end{remark}
   
\subsection{Proof of Theorem \ref{thm:ZJJ'}}   This section is devoted to a proof of Theorem \ref{thm:ZJJ'}. We keep the notations from \S\ref{sec:Springer}.
    
\subsubsection{Some subposets}

By definition, if $(v, w) \in Z_{J, J'}$, then $v \in {}^{J'} W$. For $j \in J'$, since $v \le w$ and $s_j v \neq w$, we have $w \le s_j w$.  Thus $w \in {}^{J'} W$. In particular, if $(v, w) \in Z_{J, J'}$ and $v=w_{J'} w_0$, then $w=w_{J'} w_0$. 

Let $v \in {}^{J'} W$. We define  
\begin{align*}
Z_v &= \{w \in W \vert (v, w ) \in Z_{J, J'}\};\\
P_v &= \{w \in W \vert v \le w, s_j v \not \le w\,\, \forall j \in J'\} \supset Z_v;\\
Q_v &= \{w \in W \vert v \le w,  v \not \le s_iw\,\, \forall i \in J\} \supset Z_v. 
\end{align*}
Then it follows that $Z_v  =P_v \cap Q_v$ and $Z_{J, J'}=\sqcup_{v \in {}^{J'} W} Z_v$. 

By \cite{Dyer1}*{Proposition~2.3}, there exists a reflection order $\preceq$ on $T$ such that
\begin{equation}\label{eq:ZT}
    \begin{split}
      t_1 \preceq t_2, \quad \text{for any } t_1 \in T\cap W_{J'}, t_2 \in T - W_{J'};\\
      t_1 \preceq t_2, \quad \text{for any } t_1 \in T- W_{J}, t_2 \in T \cap W_{J}.
      \end{split}
\end{equation}

This reflection order induces an EL-labeling on $[v, w_0]$. Let $M=M_{\preceq}([v,w_0])$ be the acyclic complete  matching of $G([v, w_0])$ defined in \S\ref{subsec:defM}.

\subsubsection{The $M$-subset $P_v$}\label{sec:M-P}
Suppose that $v \neq w_{J'}w_0$. By Theorem~\ref{thm:M}, $M$ gives an acyclic complete matching on $[v, w_0]$. Let $\{v_1, \dots, v_n\}$ be the set of atoms of $[v,w_0]$. Recall that $v \in {}^{J'} W$. So if $ vv_a^{-1} \in W_{J'}$, then $v_a = s_iv$ for some $i \in J'$. By the choice of the reflection order $\preceq$ on $T$ and Corollary~\ref{cor:shelling}, the set $\cup_{i \in J'} [s_iv, w_0]$ is an $M$-subset. Thus $P_v= [v, w_0] - \cup_{i \in J'} [s_iv, w_0]$ is again an $M$-subset. Therefore, $M$ gives an acyclic complete matching on $P_v$. 

\subsubsection{The $M$-subset $Q_v$}\label{sec:M-Q}

\begin{lem}\label{lem:interval}
Let $v, w \in W$ with $v \le w$. Then $\{a \in W_J \vert v \le a \, {}^Jw \}= [x, w_J]$ for some $x \in W_J$.
\end{lem}

\begin{proof} We can assume $w = w_J {}^Jw$, where $w_J$ is the longest element in $W_J$. Let $Q=\{a \in W_J \vert v \le a \, {}^Jw \}$. It is easy to see that $w_J$ is the unique maximal element in $Q$.  It remains to show that the minimal element of $Q$ is unique.
Let $x \in Q$ be a minimal element. Recall $[v,w]$ is EL-shellable with respect to the reflection order $\preceq$ on $T$. We consider the unique increasing maximal chain in $[v,x\,{}^Jw]$, denoted by $c(v,x\,{}^Jw)=v \le v_1 \le v_2 \le \cdots \le v_l \le x\,{}^Jw$. Let $c(x\,{}^Jw,w) = x\,{}^Jw \le x_1  \le \cdots \le w$ be the unique increasing maximal chain in $[x\,{}^Jw,w]$. Then we must have $xv^{-1}_{l} \not \in W_J$ by the minimality of $x$. We also have $x_1 x^{-1} \in W_J$, since we can construct the chain in $[x, w_J]$ by the convention \eqref{eq:EL}. So the combined chain 
\[
v \le v_1 \le v_2 \le \cdots \le v_l \le x\,{}^Jw \le x_1 \le \cdots \le w
\]
is an increasing maximal chain in $[v,w]$. Since this chain is unique, the element $x$ must be unique. 
\end{proof}

 For any $w \in {}^J W$, we define $Q_{v, w}=[v, w_0] \cap W_J w$. Then by definition, $Q_v=[v, w_0]-\sqcup_w Q_{v,w}$, where $w$ runs over the elements in ${}^J W$ with $Q_{v,w} \neq \{w_J w\}$. 

By the definition of the EL-labeling in \eqref{subsec:ELorder}, any edge in $Q_{v,w}$ is labeled by a reflection in $t \in W_J \cap T$. On the other hand, if an edge $\{x, y\}$ on the Hasse diagram $G([v,w_0])$ connects some $Q_{v,w} \neq Q_{v,w'}$, it will be labeled by some $t' \not \in W_J$. We have $t' \preceq t$ by the definition of the reflection order. By Lemma~\ref{lem:interval} and the isomorphism of posets $Q_{v,w} \cong Q_{v,w} w^{-1}$, the Hasse diagram of $Q_{v,w}$ is a connected graph. Hence by the definition in \S\ref{subsec:defM} and choice of the reflection order $\preceq$ in \eqref{eq:ZT}, $Q_{v, w}$ is an $M$-subset if $Q_{v,w} \neq \{w_J w\}$. Hence $Q_v$ is again an $M$-subset. Since $M$ gives an acyclic complete matching on $[v, w_0]$, $M$ also gives an acyclic complete matching on $Q_v$.

\subsubsection{The matching on $Z_{J, J'}$} 
Recall that $Z_{J, J'}=\sqcup_{v \in {}^{J'} W} Z_v$ and $Z_v=P_v \cap Q_v$. For $v \neq w_{J'} w_0$, $M$ gives an acyclic complete matching on $P_v$ and $Q_v$. Thus $M$ also gives acyclic complete matching on $Z_v$. On the other hand, if $v=w_{J'} w_0$, then $Z_v=\{(w_{ J'} w_0, w_{ J'} w_0)\}$ is a singleton. Therefore $M$ gives a matching on $Z_{J, J'}$ with the only unmatched element of $(w_{ J'} w_0, w_{ J'} w_0)$. 

It remains to show that this matching on $Z_{J, J'}$ is acyclic. Note that any arrow from $(v,w) \in Z_v$ to $(v',w') \in Z_{v'}$ for $v \neq v'$ must be of the form 
    \[
        (v,w) \rightarrow (v',w'), \quad \text{with } v \lessdot v', w = w'.
    \]
    Such arrows always increase the length of $v$, i.e., $\ell(v') = \ell(v) +1$. So we cannot have any cycle traversing $Z_v$ for different $v \in W$. Since we do not have any cycles within each $Z_v$ either, the matching has to be acyclic. 
    
The proof of Theorem~\ref{thm:ZJJ'} is complete. 


\section{Fibers on totally nonnegative flag varieties}\label{sec:fibers}
\subsection{Partial flag varieties}
 
We keep the notations in \S\ref{sec:fl}. For any $K \subset I$, we denote by $P_K^+ \supset B^+$ the corresponding standard parabolic subgroup of $G$ and $\CP_K=G/P^+_K$ the partial flag variety. Let $\pi: \CB \to \CP_K$, $g \cdot B^+ \mapsto g \cdot P^+_K$ be the natural projection map. The totally nonnegative partial flag variety $\CP_{K, \ge 0}$ is defined to be the image $\pi(\CB_{\ge 0})$. Since $\pi$ is proper, $\CP_{K, \ge 0}$ equals the Hausdorff closure of $G_{\ge 0} \cdot P^+_K$ in $\CP_K$. 

The cellular decomposition $\CB_{\ge 0}=\sqcup_{v \le w} \CB_{v, w, >0}$ induces a cellular decomposition on $\CP_{K, \ge 0}$ (see \cite{Rie06}): 
\[
     \CP_{K, \ge 0} = \sqcup_{(v,w) \in Q_K} \CP_{K, v,w, > 0}, \quad \text{ where } \CP_{K, v,w, > 0} = \pi(\CB_{v, w, >0}) \text{ by definition}.
\]
The face poset $Q_K$ is defined as $Q_K = \{(v,w) \in W \times W \vert w \in W^K, v \le w\}$, where $(v',w') \le (v,w)$ if and only if there exists $u \in W_K$ with $ v \le v'u \le w'u \le w$.

Let $(v,w) \in Q_K$. Recall that $\CB_{v, w, \ge 0}$ is the Hausdorff closure of $\CB_{v, w, >0}$. Let $\CP_{K, v, w, \ge 0}=\pi(\CB_{v, w, \ge 0})$ be the Hausdorff closure of $\CP_{K, v, w, >0}$. We denote by $\pi_{v,w} :\CB_{v, w, \ge0} \rightarrow \CP_{K,v, w, \ge 0}$  the projection obtained by restricting $\pi$.

\subsection{The fibers}\label{subsec:He0}

We first recall some results on the Coxeter group $W$. By  \cite{He}*{Lemma 1}\footnote{We use the notation $\circ_l$ and $\circ_r$ for $\triangleright$ and $\triangleleft$ in loc.cit., respectively.}, for any $x, y \in W$, 
\begin{itemize}[leftmargin=*]
	\item there exists a unique maximal element (with respect to $\le$) in $\{x' y'; x' \le x, y' \le y\}$. We denote this element by $x \ast y$. Moreover, $x \ast y=x' y=x y'$ for some $x' \le x$ and $y' \le y$ with $\ell(x \ast y)=\ell(x')+\ell(y)=\ell(x)+\ell(y')$;
	
	\item there exists a unique minimal element (with respect to $\le$) in $\{x' y; x' \le x\}$. We denote this element by $x \circ_l y$. Moreover, $x \circ_l y=x' y$ for some $x' \le x$ with $\ell(x \circ_l y)=\ell(y)-\ell(x')$;
	
	\item there exists a unique minimal element (with respect to $\le$) in $\{x y'; y' \le y\}$. We denote this element by $x \circ_r y$. Moreover, $x \circ_r y=x y'$ for some $y' \le y$ with $\ell(x \circ_r y)=\ell(x)-\ell(y')$.
\end{itemize}

Let $(v',w') \le (v,w)$ in $Q_K$. 
We define 
    \begin{align*}
    F^K_{(v',w'), (v,w)}  =  \{(a,b) \in W_K \times W_K \vert & a \le b ; v \le v'a \le w'b \le w, v'a \circ_r b^{-1} = v', \\
        & \ell(v'a) = \ell(v') +\ell(a)\}.
    \end{align*}
In particular, if $(v', w')=(v, w)$, then $F=\{(e, e)\}$ is a singleton. The set  $F^K_{(v',w'), (v,w)}$ is naturally a poset with $(a,b) \le (a',b')$ if and only if $ a\le a' \le b' \le b$ in the usual Bruhat order of $W_K$.

 \begin{prop}\label{prop:BPfiber}
  Let $(v', w') \le (v,w)$ in $Q_K$. For any $x \in \CP_{K, v',w', >0}$, we have 
    \[
         \pi_{v,w}^{-1}(x) \cong \bigsqcup_{(a,b) \in  F^K_{(v',w'), (v,w)} } \CB_{a,b, >0}.
    \] 
    In particular, the fiber $\pi_{v,w}^{-1}(x)$ is a regular CW complex with the face poset $F^K_{(v',w'), (v,w)}$. 
    \end{prop}

    \begin{proof}
     We fix a reduced expression ${\bf w'} = s_{i_1} s_{i_2} \cdots s_{i_n} $ of $w'$ and let ${\bf v'_+} = t_{i_1} t_{i_2} \cdots t_{i_n}$ be the unique positive subexpression of ${\bf w'}$ in the sense of \cite{MR}*{Lemma~3.5}. Here $t_{i_j} = 1$ or $s_{i_j}$ for any $1 \le j \le n$. Following \cite{MR}*{Definition~5.1}, we define \[
 G_{\bf v'_+, \bf w', >0}=\{g_1 g_2 \cdots g_n\vert  g_j=\dot s_{i_j}, \text{ if } t_{i_j}=1; \text{ and } g_j \in y_{i_j}(\mathbb{R}_{>0}), \text{ if } t_{i_j}=s_{i_j}.\}.
\]
Then the natural map $G_{\bf v'_+, \bf w', >0} \rightarrow \CP_{K, v',w', >0}$,  $g \mapsto g \cdot \CP_K^+$,  is an isomorphism. We write 
     \[
     x = g \cdot \CP_K^+, \qquad \text{ for the unique } g \in  G_{\bf v'_+, \bf w', >0}.
     \]
    We then define $\phi: \pi_{v,w}^{-1}(x) \rightarrow \CB$, $y \mapsto g^{-1}y$. 

Let $y \in \pi_{v,w}^{-1}(x)$ such that $y \in \CB_{p,q, >0}$ for some $v \le p \le q \le w$ in $W$. We have $q = w b$ for some $b \in W_J$. We can fix a reduced expression ${\bf q} = {\bf w'} {\bf b}$ of $q$ starting with the expression ${\bf w'}$. 

Let ${\bf p_+}$ be the unique positive subexpression of $p$ in ${\bf q}$. We write ${\bf p_+} = ({\mathbf p_1}, {\bf p_2})$ as a subexpression of $({\bf w'}, {\bf b})$. Since $\pi_{v,w}(y) = x$, we must have ${\bf p_1} = {\bf v'_+}$ and ${\bf p_2} = {\bf a_+}$ as positive expressions of some $a \in W_J$ in ${\bf b}$. 
In particular, we have 
\begin{equation}\label{eq:pq}
G_{\bf p_+, \bf q, >0} = G_{\bf v'_+, \bf w', >0} \cdot G_{\bf a_+, \bf b, >0}.
\end{equation}
We hence can write $y = g' g'' \cdot \CB^+$ for some $g' \in G_{\bf v'_+, \bf w', >0}$ and $g'' \in G_{\bf a_+, \bf b, >0}$. 
We also see that $v' a \circ_r b^{-1} = v'$  and $\ell(p) = \ell(v') + \ell(a)$, since this subexpression ${\bf p_+}$ must be rightmost and reduced/non-decreasing in the sense of \cite{MR}*{Definition~3.4}.  

Since $\pi_{v,w}(y) =x$, we obtain $g' = g$. Then we conclude that $g^{-1} y = g'' \cdot \CB^+$. This shows $\phi(y) \in \CB_{\ge0}$. More precisely, we have shown that 
\[
\phi(\pi_{v,w}^{-1}(x)) \subset \bigsqcup_{(a,b) \in  F^K_{(v',w'), (v,w)} } \CB_{a,b, >0}.
\]
On the other hand, it is easy to show that the reverse inclusion using \eqref{eq:pq}. This shows that $\pi_{v,w}^{-1}(x)$ is a regular CW complex with the face poset $F^K_{(v',w'), (v,w)}$. 
    \end{proof}

Now we state the main combinatorial result of this section. 

\begin{thm}\label{thm:FK}
For any $(v', w') \le (v, w)$ in $Q_K$, the poset $F^K_{(v',w'), (v,w)}$ admits an acyclic matching with a unique unmatched element $(\tilde{z}, \tilde{z}) \in W_K \times W_K$.
\end{thm}

Theorem~\ref{thm:FK} will be proved in \S\ref{sec:FK}. Combining Theorem~\ref{thm:FK} with Theorem~\ref{thm:Chari}, we establish the following contractibility result. 

\begin{thm}\label{thm:BPfiber}
The fibers of the projection map $\pi_{v,w} :\CB_{v, w, \ge0} \rightarrow \CP_{K,v, w, \ge 0}$ are contractible.
\end{thm}

In \S\ref{sec:regular}, we will use Theorem \ref{thm:BPfiber} to give a new and simple proof of the regularity property of $\CP_{K, \ge 0}$. This is our main motivation to study the fibers of $\pi_{v,w} :\CB_{v, w, \ge0} \rightarrow \CP_{K,v, w, \ge 0}$. It is also interesting to study the fiber of $\pi: \CB_{\ge 0} \to \CP_{\ge 0}$, which we will not pursue here.


\subsection{Reformulation of $F^K_{(v',w'), (v,w)}$}\label{subsec:FK} 
The result in this subsection holds for any Coxeter group. We first prove a counterpart of Lemma~\ref{lem:interval}.
\begin{lem}\label{lem:z'}
Let $w' \le w$ in $W$. Then $\{a \in W_K \vert w'a \le w\} = [e,z']$ for some $z' \in W_K$.
\end{lem}

\begin{proof}
The proof is similar to the proof of Lemma~\ref{lem:interval}.  
In order to match the convention used in the edge labeling in \eqref{eq:EL}, we prove the equivalent statement that $Q=\{a \in W_K \vert a (w{'})^{-1} \le w^{-1}\} = [e, x]$ for some $x \in W_K$. 

It is clear that $e$ is the unique minimal element of $Q$. Let $x$ be any maximal element of $Q$. It is clear that $[e, x ] \subset Q$.

We fix a reflection order $\preceq$ on $W$ such that $t_1 \preceq t_2$, for any $t_1 \in T \cap W_K$, $t_2 \in T- W_K$. Recall that $[(w{'})^{-1},w^{-1}]$ is EL-shellable with respect to the reflection order $\preceq$ on $T$. 

We consider the unique increasing maximal chain in $[x(w{'})^{-1}, w^{-1}]$, denoted by \linebreak$c(x(w{'})^{-1},w^{-1}) =x(w{'})^{-1} \le x_1 \le \cdots \le w$. We must have $x(w{'})^{-1} x_1^{-1} \not \in W_K$ by the maximality of $x$.
Let $c((w{'})^{-1},x(w{'})^{-1}) = (w{'})^{-1} \le y_1 \le \cdots \le y_l \le x(w{'})^{-1}$ be the unique increasing maximal chain in $[(w{'})^{-1}, x(w{'})^{-1}]$. We then have $y_l w' x^{-1} \in W_K$, since we can construct the increasing maximal chain in $[e,x] \subset W_K$ by the labeling convention \eqref{eq:EL}.

Hence the combined chain 
\[
(w{'})^{-1} \le y_1 \le \cdots \le y_l \le x(w{'})^{-1} \le x_1 \le \cdots \le w^{-1}
\]
is increasing and maximal. By the definition of the EL-labeling on  $[(w{'})^{-1}, w^{-1}]$, this chain must be unique. Therefore the element $x$ is unique. 
\end{proof}

\begin{prop}\label{prop:reF}
Let $(v', w') \le (v, w)$ in $Q_K$. Let $z  = (v')^{-1} \circ_l v$ and $z'$ be the element defined in Lemma~\ref{lem:z'}. Set $N_R(v') = \{t \in T \vert v't<v'\}$. Then \begin{align*}
F^K_{(v',w'), (v,w)} & =\{(a,b) \in W_K \times W_K \vert   z  \le a \le b  \le z', v'a \circ_r b^{-1} = v'\}\\
&= \{(a,b) \in W_K \times W_K \vert z  \le a \le b  \le z', ta \not \le b \,\, \forall t \in N_{R}(v')\}\\
&= \{(a,b) \in W_K \times W_K \vert z  \le a \le b  \le z', ta \not \le b \,\, \forall t \in N_{R}(v')\text{ with }a \lessdot ta\}.
\end{align*}
\end{prop}

\begin{proof}
Note that $z \in W_K$ by the definition of the partial ordering on $Q_K$. 
By \cite{HL}*{Lemma~4.3}, we have $v \le v' a$ if and only if $(v')^{-1} \circ_l v \le a$ and $w'b \le w$ if and only if $b \le (w')^{-1} \circ_l w$. Moreover, if $a \le b$ and $v'a \circ_r b^{-1} = v'$, then $v' \ge v'a \circ_r a^{-1} \ge v'a \circ_r b^{-1}= v'$. Therefore $v'a \circ_r a^{-1}=v'$. By definition, there exists $a' \le a$ such that $v'=v' a (a') \i$ and $\ell(v'a)-\ell(a')=\ell(v')$. Hence $a'=a$ and $\ell(v' a)=\ell(v')+\ell(a)$. Therefore we conclude that $$F^K_{(v',w'), (v,w)} = \{(a,b) \in W_K \times W_K \vert   z  \le a \le b  \le z'; v'a \circ_r b^{-1} = v'\}.$$

Note that if $ta  \le b$ for some $t \in N_R(v')$, then $v'a \circ_r b^{-1} \le v'aa^{-1}t = v't < v'$. Thus we have 
$$
F^K_{(v',w'), (v,w)} \subset  \{(a,b) \in W_K \times W_K \vert z  \le a \le b  \le z', ta \not \le b \,\, \forall t \in N_{R}(v')\}.
$$
It remains to prove that if $(a,b) \in W_K \times W_K$ with $z \le a \le b  \le z'$ and $ ta \not \le b$ for all $t \in N_{R}(v')$ with $a \lessdot ta$, we have $F^K_{(v',w'), (v,w)}$.

Assume the contrary that $v'a \circ_r b^{-1}< v'$. Let $b = s_{i_1} \cdots s_{i_n}$ be a reduced expression of $b$. We write $v_n = v'a \circ_r s_{i_n}$ and $v_{i} =v_{i+1} \circ_r s_{i_{i}}$ inductively. Fix a reduced expression $v'a = s_{k_1} \cdots s_{k_l} s_{k_{l+1}} \cdots s_{k_{l+l'}}$, such that $v' =s_{k_1} \cdots s_{k_l}$ and $a= s_{k_{l+1}} \cdots s_{k_{l+l'}}$. Then by the deletion property of Coxeter groups, we have either  $v_{i} =v_{i+1}$ or $v_{i}$ is a subexpression of  $v_{i+1}$ by deleting a unique simple reflection in the (fixed) reduced expression. The resulting expression is always reduced by the definition of $- \circ_r s_{i_{i}}$. Since $a \le b$, the resulting expression of $v_1 = v'a \circ_r b^{-1}$ must be a subexpression of $v'$. By the assumption, $v'a \circ_r b^{-1}< v'$. Suppose that $v_{j-1}$ is the first element obtained from $v_j$ by deleting a simple reflection in $s_{k_1} \cdots s_{k_l}$. In other words, $j$ is the largest integer such that $v_{j} = v' a_{j}$ with $a_{j }  \le a$ and $v_{j} \circ_r s_{i_{j}} = v'_{j} a_{j}$ for $v'_{j} \lessdot v'$. 
 
By the definition of $\circ_r$, we have $\ell(a) = \ell(a_{j}) + \ell(a_{j}^{-1} a)$. We also have that $v'_j a_j=v' a_j s_{i_j}=v' (a_j s_{i_j} a_j \i) a_j$ and $v'_j=v' (a_j s_{i_j} a_j \i)$. Thus $a_{j} s_{i_j} a_{j}^{-1} \in N_{R}(v')$. Note that $a_{j} s_{i_j} a_{j}^{-1}$ is a reflection and we have either $a<a_{j} s_{i_j} a_{j}^{-1} a$ or $a>a_{j} s_{i_j} a_{j}^{-1} a$. If $a>a_{j} s_{i_j} a_{j}^{-1} a$, then $v' a=(v' a_{j} s_{i_j} a_{j}^{-1}) (a_{j} s_{i_j} a_{j}^{-1} a)$ and $\ell(v')+\ell(a)=\ell(v' a) \le \ell(v' a_{j} s_{i_j} a_{j}^{-1})+\ell(a_{j} s_{i_j} a_{j}^{-1} a)<\ell(v')+\ell(a)$. That is a contradiction. Hence $a<a_{j} s_{i_j} a_{j}^{-1} a=a_j s_{i_j} (a_j \i a)$. We have $\ell(a)<\ell(a_j s_{i_j} (a_j \i a)) \le \ell(a_j)+1+\ell(a_j \i a)=\ell(a)+1$. Hence we conclude that $\ell(a_j s_{i_j} (a_j \i a))=\ell(a)+1$ and $a \lessdot a_{j} s_{i_j} (a_{j}^{-1} a)$. 

Since $v' a \circ_r b \i=v' a_j \circ_r (s_{i_1} \cdots s_{i_j}) \i <v'$. By definition, we have $a_j< s_{i_1} \cdots s_{i_j}$. Also $a_j s_{i_j}>a_j$. Thus $a_j s_{i_j} \le s_{i_1} \cdots s_{i_j}$. So $a_j s_{i_j} (a_j \i a) \le s_{i_1} \cdots s_{i_j} s_{i_{j+1}} \cdots s_{i_n}=b$. This contradicts the assumption that $ta \not \le b$ for all $t \in N_{R}(v')$ with $a \lessdot ta$. Hence $v'a \circ_r b^{-1}=v'$. 

The proposition is proved. 
\end{proof}

\begin{rem}
Thanks to the reformulation in Proposition~\ref{prop:reF}, one may regard the poset $F^K_{(v',w'), (v,w)}$ as a generalization of $Z_{\emptyset,J'}$ considered in \eqref{eq:ZJJ'}.
\end{rem}

\begin{cor}\label{lem:fbclosed}
Let $(a,b) \in F^K_{(v',w'), (v,w)} $. Then for any $(a',b') \in W_K \times W_K$ with $a \le a' \le b' \le b$, we have $(a',b') \in F^K_{(v',w'), (v,w)}$.
\end{cor}

\begin{proof}
In the case where $W$ is the Weyl group of a Kac-Moody group, this statement follows from Proposition~\ref{prop:BPfiber} and the closeness of the fibers of $\pi_{v, w}$. We give a combinatorial proof here, which applies to arbitrary Coxeter groups.


By Proposition~\ref{prop:reF}, it suffices to show that $v' a' \circ_r (b') \i=v'$ for any $a \le a' \le b' \le b$.

Since $a' \le b'$, by definition we have $v'a' \circ_r (b')^{-1} \le v'$. We prove the other direction. By \cite{He}*{Lemma~2}, we have $$v' =v'a \circ_r (b)^{-1} \le (v' \ast a') \circ_r (b')^{-1} \le (v' \ast a') \circ_r (a')^{-1} \le  v'.$$ In particular, we have $(v' \ast a') \circ_r (b')^{-1}=(v' \ast a') \circ_r (a')^{-1}=v'$. We deduce from $(v' \ast a') \circ_r (a')^{-1}=v'$ that $v' a'=v' \ast a'$. Thus $(v' a') \circ_r (b')^{-1}=v'$. The statement is proved. 
\end{proof}

\subsection{Proof of Theorem \ref{thm:FK}}\label{sec:FK}
By \cite{Dyer1}*{Lemma~2.7}, there exists an reflection order $\preceq$ on $T$ such that 
 \begin{equation}\label{eq:refv'}
   t_1 \preceq t_2 \text{ for any } t_1 \in N_R(v') \text{ and } t_2 \in T - N_R(v').
  \end{equation}
Let $M = M_{\preceq}([a, z'])$ be the acyclic complete matching constructed in \S\ref{subsec:M}. We show $M$ gives the desired matching on $F^K_{(v',w'), (v,w)}$. The strategy of the proof is similar to Theorem~\ref{thm:ZJJ'} thanks to the reformulation of the poset $F^K_{(v',w'), (v,w)}$ in Proposition~\ref{prop:reF}. However, we need a detailed analysis on the generalized quotients of intervals, which we will discuss below. 

Let $z, z'$ be as in Proposition~\ref{prop:reF}. Let $\{v'\}\backslash[z,z']=  \{a \in [z,z'] \vert \ell(v'a) = \ell(v') + \ell(a) \}$ be the (left) generalised quotient of the interval $[z,z']$ in the sense of \cite{BW}.\footnote{Note that we consider the left quotient here, while the right quotient was considered in \cite{BW}. There is clearly no essential difference.} Note that $a \in \{v'\}\backslash[z,z']$ if and only if $(a,a) \in F^K_{(v',w'), (v,w)}$.  

Let $a,a' \in \{v'\}\backslash[z,z']$. Let $u$ be a minimal upper bound of $a,a'$ in $[z,z']$ (hence also in $W_K$). Then by \cite{BW}*{Theorem~3.7}, we see that $\ell(v'u) = \ell(v') +\ell(u)$. Since the interval $[z,z']$ is finite, we obtain a unique maximal element in $\{v'\}\backslash[z,z']$ by iterating the process. We denote this element by $\tilde z$. We define 
\begin{align*} P_a &= \{ b \in W_K \vert a \le b \le z', v'a \circ_r b^{-1} = v'\} \\ &= \{b \in W_K \vert a \le b \le z', ta \not \le b\,\, \forall t \in N_R(v')\}.\end{align*} 

{\it (a)  We have $P_a = \{a\}$ if and only if $a = \tilde{z}$.}

It is easy to see that $P_{\tilde{z}}=\{\tilde z\}$. Suppose that $a<\tilde{z}$. By \cite{BW}*{Theorem~3.4}, the generalized quotient $\{v'\}\backslash[z, z']$ is a graded poset with grading given by the length function of $W$. Let $a' \in \{v'\}\backslash[z, z']$ be such that $a \lessdot a' \le \tilde{z}$. 

We show $v'a \circ_r (a')^{-1} = v'$. Assume the contrary that $v'a \circ_r (a')^{-1} = u < v'$. Since $\ell(u) \ge \ell(v'a) - \ell(a') = \ell(v) -1$, we have $v'a (a')^{-1} = u=v't$ for some $t \in N_{R}(v')$. Therefore $ta = a'$ and $\ell (v'a') = \ell(v'ta) < \ell(v') + \ell(a)$. This contradicts the fact that $a' \in \{v'\}\backslash[z, z']$. Thus $P_a \supset \{a, a'\}$. (a) is proved. 

 {\it (b) The poset $P_a = \{ b \in W_K \vert a \le b \le z',ta \not \le b\,\, \forall t \in N_R(v') \}$ has a complete acyclic matching for any $a \in \{v'\}\backslash[z,z']$ unless $a = \tilde{z}$.}

Assume $a < \tilde{z} \le z'$. Let $a_1, a_2, \dots, a_s$ be all the atoms of $[a,z']$ such that $a_i a^{-1} \in N_R(v')$. Then, similarly to \S\ref{sec:M-P}, we see that $[a,z'] - \cup_{1 \le i \le s}[a_i,z']$ is an $M$-subset due to the choice of the reflection order $\preceq$ in \eqref{eq:refv'}. Recall if $b \in W_K$ with $v'a \circ_r b^{-1} <v'$, then $a \lessdot ta \le b$ for some $t \in N_{R}(v')$ by Proposition~\ref{prop:reF}. So $b \in \cup_{1 \le i \le s}[a_i,z']$. Therefore $P_a = [a,z'] - \cup_{1 \le i \le s}[a_i,z']$ has a complete acyclic matching unless $a=\tilde z$.

The rest of the proof is entirely similar to that of Theorem~\ref{thm:ZJJ'}.

\subsection{Regularity theorem for partial flag varieties} \label{sec:regular}

The regularity theorem for totally nonnegative partial flag varieties was first established for reductive groups in \cite{GKL} and later for arbitrary Kac-Moody groups in \cite{BH22}. In both cases, the proofs involve the compatibility of the totally positivity on partial flag varieties with certain atlas maps, which are rather difficult to establish; see, for example, \cite{BH22}*{\S6-\S8}. On the other hand, the regularity theorem for totally nonnegative full flag varieties is dramatically easier to establish based on the product structure introduced in \cite{BH22}*{\S5.1}.

In this section, we give a new proof of the regularity theorem for totally nonnegative partial flag varieties, assuming the regularity theorem for the full flag varieties.  We refer to \cite{DHM}*{\S2.1} for the definition of CW complexes.

We first recall the following topological result.
\begin{prop}\label{lem:cellquotient}\cite{DHM}*{Corollary~2.33}
Let $B^n$ be a  $n$-dimensional closed ball and $S^{n-1}$ be the boundary sphere of $B^n$. Let $\sim$ be an equivalent relation on the close ball $B^n$ such that 
\begin{enumerate}
    \item all equivalently classes are contractible ;
    \item $S^{n-1}/ \sim$ is homeomorphisc to $S^{n-1}$;
    \item if $x \sim y$ with $x \in S^{n-1}$, then $y \in S^{n-1}$;
    \item if $x \sim y$ with $ x \not \in S^{n-1}$, then $y = x$.
\end{enumerate}
Then $B^n$ is homeomorphic to $B^n / \sim$.
\end{prop}

Let $Q = \{(v,w) \vert v \le w \text{ in } W\}$. We equip $Q$ with a partial order $\le$ such that $(a,b) \le (v,w)$ if and only if $v \le a \le b \le w$ in the Bruhat order of $W$.

Recall \S\ref{sec:fl} that $\CB_{v,w, \ge 0}$ is a regular CW complex homeomorphic to a closed ball $B^{\ell(w) - \ell(v)}$ for any $(v,w) \in Q$.  By \cite{BH22}*{Theorem~5.2}, the boundary sphere of $ B^n \cong \CB_{v,w, \ge 0}$ is precisely $S^{n-1} \cong \sqcup_{(v',w') < (v,w)} \CB_{v',w', >0}$.
 We explain how to deduce that $\CP_{K,v,w, \ge 0}$ is a regular CW complex homeomorphic to a closed ball. 

The natural projection $\pi_{v,w}: \CB_{v,w, \ge 0} \rightarrow \CP_{K, v,w, \ge 0}$ restricts to a homeomorphism $ \CB_{v,w, > 0} \cong \CP_{K, v,w, > 0}$ for any $(v,w) \in Q_K$. Moreover, $\pi_{v,w}$ maps $S^{\ell(w) - \ell(v) -1} \cong \sqcup_{(a,b)< (v,w) \in  Q}\CB_{a,b, > 0}$ to $\sqcup_{(v',w') < (v,w) \in Q_K }\CP_{K, v',w', > 0}$. 

 Since the choice of $(v,w) \in Q_K$ is arbitrary, this shows that $\CP_{K, v,w, \ge 0}$ is a CW complex with cells $\CP_{K, v',w', > 0}$ for $(v',w') \le (v,w) \in Q_K$ with characteristic maps $\pi_{v',w'}: \CB_{v',w', \ge0} \rightarrow \CP_{K, v',w', \ge 0}$.

It remains to show that $\CP_{K, v',w', \ge 0}$ is homeomorphic to a closed ball for any $(v',w') \in Q_K$. Without loss of generality, we show that $\CP_{K, v,w, \ge 0}$ is homeomorphic to a closed ball by induction on $n = \ell(w) - \ell(v)$. The base case $n = 0$ is trivial. 

The projection $\pi_{v,w}$ factors through the natural quotient
        \[
        \begin{tikzcd}
             \CB_{v,w, \ge 0} \ar[r, "\pi_{v,w}"] \ar[d] & \CP_{K, v,w, \ge 0} \\
             \CB_{v,w, \ge 0} / \sim \ar[ru, "\tilde{\pi}_{v,w}" below right]&
        \end{tikzcd}
        \]
        Since $\pi_{v,w}$ is proper, hence close, $\tilde{\pi}_{v,w}$ is also a closed map. Then we see that $\tilde{\pi}_{v,w}$ is a homeomorphism, since it is a bijective continuous closed map. 

    We next show $(\CB_{v,w, \ge 0} / \sim )
            \cong \CB_{v,w, \ge 0} \cong B^{\ell(w) - \ell(v)}$ following Lemma~\ref{lem:cellquotient} and the induction hypothesis. 
             
            The equivalent classes of $\sim$ are precisely the fibers of $\pi_{v,w}$, whose contractibility follows from Theorem~\ref{thm:BPfiber}. This verifies the condition (1) in Proposition~\ref{lem:cellquotient}.  
            
           The boundary sphere of $ B^n \cong \CB_{v,w, \ge 0}$ is precisely $
            S^{n-1} \cong \sqcup_{(v',w') < (v,w)} \CB_{v',w', >0}$ by \cite{BH22}*{Theorem~5.2}. 
            Therefore, we have 
            \[
            (S^{n-1} / \sim) \cong \pi_{v,w}(S^{n-1}) =  \bigsqcup_{(v',w') < (v,w) \in Q_K} \CP_{K, v',w', >0}.
            \]
            By the induction hypothesis, each closure $\overline{\CP_{K, v',w', >0}} = \CP_{K, v',w', \ge0}$ is a regular CW complex homeomorphic to a closed ball. Hence $\sqcup_{(v',w') < (v,w) \in Q_K} \CP_{K, v',w', >0}$ is a regular CW complex. We then see that it is a sphere of dimension $n-1$ by the combinatorics of the face poset following \cite{Bj2}. Hence $( S^{n-1}/ \sim) \cong  S^{n-1}$ verifying condition (2) in Proposition~\ref{lem:cellquotient}.
            
           Since $\pi_{v,w}: \CB_{v,w, \ge 0}  \rightarrow \CP_{K,v,w, \ge 0} $ is the characteristic map of a CW complex. Both conditions (3) and (4) in Proposition~\ref{lem:cellquotient} follow immediately.

           We have now verified all the conditions in Proposition~\ref{lem:cellquotient}. We conclude that $\CP_{v,w, \ge 0} $ is a regular CW complex homeomorphic to $B^{\ell(w) - \ell(v)}$.



\begin{thebibliography}{}

\bib{AHBC16}{book}{
   author={Arkani-Hamed, Nima},
   author={Bourjaily, Jacob},
   author={Cachazo, Freddy},
   author={Goncharov, Alexander},
   author={Postnikov, Alexander},
   author={Trnka, Jaroslav},
   title={Grassmannian geometry of scattering amplitudes},
   publisher={Cambridge University Press, Cambridge},
   date={2016},
   pages={ix+194},
   isbn={978-1-107-08658-6},
   review={\MR{3467729}},
   doi={10.1017/CBO9781316091548},
}

\bib{Bj1}{article}{
   author={Bj\"{o}rner, Anders},
   title={Shellable and Cohen-Macaulay partially ordered sets},
   journal={Trans. Amer. Math. Soc.},
   volume={260},
   date={1980},
   number={1},
   pages={159--183},
   issn={0002-9947},
   review={\MR{0570784}},
   doi={10.2307/1999881},
}

\bib{Bj2}{article}{
   author={Bj\"{o}rner, Anders},
   title={Posets, regular CW complexes and Bruhat order},
   journal={European J. Combin.},
   volume={5},
   date={1984},
   number={1},
   pages={7--16},
   issn={0195-6698},
   review={\MR{0746039}},
   doi={10.1016/S0195-6698(84)80012-8},
}

\bib{BB}{book}{
   author={Bj\"{o}rner, Anders},
   author={Brenti, Francesco},
   title={Combinatorics of Coxeter groups},
   series={Graduate Texts in Mathematics},
   volume={231},
   publisher={Springer, New York},
   date={2005},
   pages={xiv+363},
   isbn={978-3540-442387},
   isbn={3-540-44238-3},
   review={\MR{2133266}},
}

\bib{Brenti}{article}{
   author={Brenti, Francesco},
   title={The intersection cohomology of Schubert varieties is a
   combinatorial invariant},
   journal={European J. Combin.},
   volume={25},
   date={2004},
   number={8},
   pages={1151--1167},
   issn={0195-6698},
   review={\MR{2095476}},
   doi={10.1016/j.ejc.2003.10.011},
}

\bib{BCM}{article}{
   author={Brenti, Francesco},
   author={Caselli, Fabrizio},
   author={Marietti, Mario},
   title={Special matchings and Kazhdan-Lusztig polynomials},
   journal={Adv. Math.},
   volume={202},
   date={2006},
   number={2},
   pages={555--601},
   issn={0001-8708},
   review={\MR{2222360}},
   doi={10.1016/j.aim.2005.01.011},
}

\bib{BH21}{article}{
    author={Bao, Huanchen},
    author={He, Xuhua},
     TITLE = {Flag manifolds over semifields},
   JOURNAL = {Algebra Number Theory},
    VOLUME = {15},
      YEAR = {2021},
    NUMBER = {8},
     PAGES = {2037--2069},
      ISSN = {1937-0652},
       doi = {10.2140/ant.2021.15.2037},
}

\bib{BH22}{article}{
   author={Bao, Huanchen},
   author={He, Xuhua},
   title={Product structure and regularity theorem for totally nonnegative flag varieties},
    publisher = {arXiv},
  year = {2022},
  eprint={https://arxiv.org/abs/2203.02137}
}

\bib{BW}{article}{
   author={Bj\"{o}rner, Anders},
   author={Wachs, Michelle L.},
   title={Generalized quotients in Coxeter groups},
   journal={Trans. Amer. Math. Soc.},
   volume={308},
   date={1988},
   number={1},
   pages={1--37},
   issn={0002-9947},
   review={\MR{0946427}},
   doi={10.2307/2000946},
}

 \bib{Chari}{article}{
   author={Chari, Manoj K.},
   title={On discrete Morse functions and combinatorial decompositions},
   language={English, with English and French summaries},
   note={Formal power series and algebraic combinatorics (Vienna, 1997)},
   journal={Discrete Math.},
   volume={217},
   date={2000},
   number={1-3},
   pages={101--113},
   issn={0012-365X},
   review={\MR{1766262}},
   doi={10.1016/S0012-365X(99)00258-7},
}

\bib{Dyd02}{article}{
   author={Dydak, Jerzy},
   title={Cohomological dimension theory},
   conference={
      title={Handbook of geometric topology},
   },
   book={
      publisher={North-Holland, Amsterdam},
   },
   isbn={0-444-82432-4},
   date={2002},
   pages={423--470},
   review={\MR{1886675}},
}

\bib{Dyer1}{article}{
   author={Dyer, Matthew},
   title={Hecke algebras and shellings of Bruhat intervals},
   journal={Compositio Math.},
   volume={89},
   date={1993},
   number={1},
   pages={91--115},
   issn={0010-437X},
   review={\MR{1248893}},
}


\bib{DHM}{article}{
   author={Dvis, James F.},
   author={Hersh, Patricia},
   author={Miller, Ezra},
   title={Fibers of maps to totally nonnegative spaces},
    publisher = {arXiv},
  year = {2019},
  eprint={https://arxiv.org/abs/1903.01420}
}




\bib{Forman}{article}{
   author={Forman, Robin},
   title={Morse theory for cell complexes},
   journal={Adv. Math.},
   volume={134},
   date={1998},
   number={1},
   pages={90--145},
   issn={0001-8708},
   review={\MR{1612391}},
   doi={10.1006/aima.1997.1650},
}

 
    \bib{FZ}{article}{
   author={Fomin, Sergey},
   author={Zelevinsky, Andrei},
   title={Cluster algebras. I. Foundations},
   journal={J. Amer. Math. Soc.},
   volume={15},
   date={2002},
   number={2},
   pages={497--529},
   issn={0894-0347},
   review={\MR{1887642}},
   doi={10.1090/S0894-0347-01-00385-X},
} 
\bib{GKL}{article}{
   author={Galashin, Pavel},
   author={Karp, Steven N.},
   author={Lam, Thomas},
   title={Regularity theorem for totally nonnegative flag varieties},
   journal={J. Amer. Math. Soc.},
   volume={35},
   date={2022},
   number={2},
   pages={513--579},
   issn={0894-0347},
   review={\MR{4374956}},
   doi={10.1090/jams/983},
}

\bib{He}{article}{
   author={He, Xuhua},
   title={A subalgebra of 0-Hecke algebra},
   journal={J. Algebra},
   volume={322},
   date={2009},
   number={11},
   pages={4030--4039},
   issn={0021-8693},
   review={\MR{2556136}},
   doi={10.1016/j.jalgebra.2009.04.003},
}


\bib{HL}{article}{
   author={He, Xuhua},
   author={Lu, Jiang-Hua},
   title={On intersections of certain partitions of a group
   compactification},
   journal={Int. Math. Res. Not. IMRN},
   date={2011},
   number={11},
   pages={2534--2564},
   issn={1073-7928},
   review={\MR{2806588}},
   doi={10.1093/imrn/rnq169},
}

\bib{HL1}{article}{
   author={He, Xuhua},
   author={Lusztig, George},
   title={Total positivity and conjugacy classes},
    publisher = {arXiv},
  year = {2022},
  eprint={https://arxiv.org/abs/2201.12479}
}

\bib{Lam15}{article}{
   author={Lam, Thomas},
   title={Totally nonnegative Grassmannian and Grassmann polytopes},
   conference={
      title={Current developments in mathematics 2014},
   },
   book={
      publisher={Int. Press, Somerville, MA},
   },
   isbn={978-1-57146-313-5},
   isbn={1-57146-313-5},
   date={2016},
   pages={51--152},
   review={\MR{3468251}},
}

\bib{MR}{article}{
   author={Marsh, R. J.},
   author={Rietsch, K.},
   title={Parametrizations of flag varieties},
   journal={Represent. Theory},
   volume={8},
   date={2004},
   pages={212--242},
   review={\MR{2058727}},
   doi={10.1090/S1088-4165-04-00230-4},
}

\bib{Jones}{article}{
   author={Jones, Brant C.},
   title={An explicit derivation of the M\"{o}bius function for Bruhat
   order},
   journal={Order},
   volume={26},
   date={2009},
   number={4},
   pages={319--330},
   issn={0167-8094},
   review={\MR{2591858}},
   doi={10.1007/s11083-009-9128-6},
}

\bib{Lu94}{article}{
   author={Lusztig, G.},
   title={Total positivity in reductive groups},
   conference={
      title={Lie theory and geometry},
   },
   book={
      series={Progr. Math.},
      volume={123},
      publisher={Birkh\"{a}user Boston, Boston, MA},
   },
   isbn={0-8176-3761-3},
   date={1994},
   pages={531--568},
   review={\MR{1327548}},
   doi={10.1007/978-1-4612-0261-5\_20},
}

\bib{Lu19}{article}{
    AUTHOR = {Lusztig, G.},
     TITLE = {Total positivity in reductive groups, {II}},
   JOURNAL = {Bull. Inst. Math. Acad. Sin. (N.S.)},
    VOLUME = {14},
      YEAR = {2019},
    NUMBER = {4},
     PAGES = {403--459},
      ISSN = {2304-7909},
       DOI = {10.21915/bimas.2019402},
       URL = {https://doi.org/10.21915/bimas.2019402},
}

    \bib{LuSp}{article}{
   author={Lusztig, G.},
   title={Total positivity in Springer fibres},
   journal={Q. J. Math.},
   volume={72},
   date={2021},
   number={1-2},
   pages={31--49},
   issn={0033-5606},
   review={\MR{4271379}},
   doi={10.1093/qmathj/haaa021},
}

\bib{Rie99}{article}{
   author={Rietsch, Konstanze},
   title={An algebraic cell decomposition of the nonnegative part of a flag
   variety},
   journal={J. Algebra},
   volume={213},
   date={1999},
   number={1},
   pages={144--154},
   issn={0021-8693},
   review={\MR{1674668}},
   doi={10.1006/jabr.1998.7665},
}

\bib{Rie06}{article}{
   author={Rietsch, Konstanze},
   title={Closure relations for totally nonnegative cells in $G/P$},
   journal={Math. Res. Lett.},
   volume={13},
   date={2006},
   number={5-6},
   pages={775--786},
   issn={1073-2780},
   review={\MR{2280774}},
   doi={10.4310/MRL.2006.v13.n5.a8},
}

\bib{RW}{article}{
   author={Rietsch, Konstanze},
   author={Williams, Lauren},
   title={Discrete Morse theory for totally non-negative flag varieties},
   journal={Adv. Math.},
   volume={223},
   date={2010},
   number={6},
   pages={1855--1884},
   issn={0001-8708},
   review={\MR{2601003}},
   doi={10.1016/j.aim.2009.10.011},
}



	\end{thebibliography}
\end{document}